\documentclass[9pt,twocolumn]{IEEEtran}
\usepackage{amsmath,amssymb,euscript,yfonts,psfrag,latexsym,dsfont,graphicx,bbm,color,amstext,wasysym,pdfsync}
\usepackage{epstopdf, bm}
\usepackage{amsmath, amssymb, psfrag, bm, latexsym, color, amstext, mathdots}
\usepackage{graphicx}
\usepackage{epstopdf, subfigure}
\usepackage{marginnote}
\usepackage{flushend}

\newcommand{\ignore}[1]{}

\newcommand{\Span}{{\rm span}}

\newcommand{\mT}{{\mathbb T}}
\newcommand{\mZ}{{\mathbb Z}}

\begin{document}
\newtheorem{thm}{Theorem}
\newtheorem{cor}[thm]{Corollary}
\newtheorem{lemma}[thm]{Lemma}
\newtheorem{prop}[thm]{Proposition}
\newtheorem{problem}[thm]{Problem}
\newtheorem{remark}{Remark}
\newtheorem{defn}[thm]{Definition}
\newtheorem{ex}[thm]{Example}

\newcommand{\gi}{g_{\rm i}}
\newcommand{\go}{g_{\rm o}}
\newcommand{\hi}{h_{\rm i}}
\newcommand{\ho}{h_{\rm o}}
\newcommand{\blambda}{{\bm \lambda}}
\newcommand{\mE}{\mathbb E}

\newcounter{tempEquationCounter}
\newcounter{thisEquationNumber}
\newenvironment{floatEq}
{\setcounter{thisEquationNumber}{\value{equation}}\addtocounter{equation}{1}
\begin{figure*}[!t]
\normalsize\setcounter{tempEquationCounter}{\value{equation}}
\setcounter{equation}{\value{thisEquationNumber}}
}
{\setcounter{equation}{\value{tempEquationCounter}}
\hrulefill\vspace*{4pt}
\end{figure*}

}
\def\spacingset#1{\def\baselinestretch{#1}\small\normalsize}

\title{The r\^ole of the time-arrow in mean-square\\
estimation of stochastic processes}

\author{Yongxin Chen, Johan Karlsson, and Tryphon T.\ Georgiou  \thanks{Supported by
Swedish Research Council, NSF under Grant ECCS-1509387, the AFOSR under Grants FA9550-12-1-0319 and FA9550-15-1-0045, and the Vincentine Hermes-Luh Endowment.
}  \thanks{
Y.~Chen and T.T.~Georgiou are with the
Department of Electrical Engineering, University of Minnesota,
Minneapolis, Minnesota 55455, USA; {\{chen2468,tryphon\}@ece.umn.edu.} J.~Karlsson is with the Department of Mathematics, KTH Royal Institute of Technology, Stockholm, Sweden; {
johan.karlsson@math.kth.se}.} }
\spacingset{1}

\maketitle

{\abstract The purpose of this paper is to explain a certain dichotomy between the information that the past and future values of a {\em multivariate} stochastic process carry about the present. More specifically, vector-valued, second-order stochastic processes may be deterministic in one time-direction and not the other. This phenomenon, which is absent in scalar-valued processes, is deeply rooted in the geometry of the shift-operator.
The exposition and the examples we discuss are based on the work of Douglas, Shapiro and Shields on cyclic vectors of the backward shift and relate to classical ideas going back to Wiener and Kolmogorov. We focus on rank-one stochastic processes for which we present a characterization of  all regular processes that are deterministic in the reverse time-direction.
The paper builds on examples and the goal is to provide pertinent insights to a control engineering audience.
}

\section{introduction}

The variance of the error in one-step-ahead prediction of a {\em scalar}, second-order, stationary, discrete-time stochastic processe is given by a well-known formula due to G.~Szeg\"o \cite{GreSze58} as the geometric mean
\begin{align}\label{eq:szego}
\exp\left\{\frac{1}{2\pi}\int_{-\pi}^\pi\log(\Phi(\theta))d\theta\right\}
\end{align}
of its power spectral density $\Phi(\theta)$.
Past and future of the process contain the same information about the present and the identical same formula provides the variance of the ``postdiction'' error when the present is estimated from future values. This is rather evident since \eqref{eq:szego} contains no manifestation of the time arrow. There is no such formula for the covariance matrix of the prediction or the postiction error for multivariable processes. The closest to such a formula was given when Wiener and Masani  \cite{WieMas57} expressed the determinant of the error covariance, herein denoted by $\Omega$, in terms of the determinent of the power spectrum,
\begin{align}\label{eq:wm}
\det(\Omega)=\exp\left\{\frac{1}{2\pi}\int_{-\pi}^\pi\log(\det(\Phi(\theta)))d\theta\right\}.
\end{align}
In a subtle way, when $\det(\Omega)=0$, this formula leaves out the possibility of a dichotomy between past and future, and as it turns out, this is indeed the case; it is perfectly possible for a (multivariable) stationary Gaussian stochastic process to be purely deterministic in one time-direction and not in the other.

Naturally, this issue has been duly noted in classical works in prediction theory where it has been pointed out that the information contained in the remote past and the information contained in the remote future may differ, see e.g., \cite[Section 4.5]{LinPic15}. Thus, the main objective of the present work is to highlight and elucidate this phenomenon with examples that are intuitively clear to an engineering audience.

More broadly, the manifestation of the time-arrow in engineering and physics is hardly a new issue, yet it is one that is not very well understood. The paradox of the apparent directionality of physics originating in physical laws that are time-symmetric is a key conundrum; Feynman states that there is a fundamental law which says, that "uxels only make wuxels and not vice versa," but we have not found this yet. Thus, the time reversibility of physical models, as well as the lack, are of great interest, see, e.g., \cite{GeoLin14, SanDelDoy11, GeoMal10}. In a similar vain, it is expected that the time-arrow will draw increasingly more attention in modeling of engineering systems as well.

Turning to time-series, the possible ways in which the time-arrow is encoded in the statistics have also been studied in the physics literature as well, see e.g., \cite{Ton90}. It is widely thought that the time-direction and ``nonlinearities'' are revealed by considering several-point correlations and higher order statistics. While this may be so at times, it is surprising to most that the time-arrow may be clearly discerned in second-order stationary processes as well, in that their predictability properties may dramatically differ depending on the time-direction. The reason that this observation is often missed (cf.\ \cite{DikVanTakDeg95, Wei75}) may be due to the fact that it is exclusively a phenomenon of vector-valued processes. Below, we explain this point with an example of a vector-valued moving-average process constructed so that the prediction error differs substantially in the two time-directions (Section \ref{sec:MA}). A limit case for a stochastic process with infinite memory allows for the process to be deterministic in one of the two time-directions but not in the other - we highlight this with an example as well.

Prediction theory of second-order processes overlaps with that of analytic functions on the unit disc and the shift operator. Thus, the exposition and technical results of the paper rely heavily on this connection and on the work of Douglas, Shapiro and Shields \cite{DouShaShi70} who obtained a characterization of cyclic vectors of the ``backward shift.'' Our analysis and examples include processes generated by filters whose transfer functions are cyclic with respect to the backward shift, or in a time-symmetric situation, processes generated by suitable acausal filters that are predictable from the infinite remote past. Besides explaining the dichotomy between past and future, and on how this relates to factorizability of the power spectrum \cite{bloomfield1983characterizations, Roz60}, we also study regular rank-one processes and explicitly characterize all such processes that are deterministic in the reverse-time direction.


The present paper is structured as follows.
In Section \ref{sec:prel} we remind the reader about connections between second-order stochastic processes and analytic functions.
We define cyclic vectors and recall key results from \cite{DouShaShi70,Hof07,Hel13}. In Section \ref{sec:MA} we present an example of a moving-average process and we compare the corresponding predictor and postdictor error covariances. In Section \ref{sec:nonreverse} we present an example that is non-deterministic in one time-direction but not so in the other. In Section~\ref{sec:CharRank1} we characterize the rank-one stochastic processes that have this property. We conclude with a discussion on factorizability and decompositions for this class of processes.

\newcommand{\X}{{X}}
The notation used in this paper is now briefly defined. We let $L_2$ be the space of square-integrable functions on the unit circle $\mT$, $H_2$ be the Hardy space of functions in $L_2$ whose negative Fourier coefficients vanish, $l_2$ be the space of square summable sequences, and let $\|\cdot\|_2$ denote the norm in the respective spaces. The orthogonal complement of $H_2$ in $L_2$ is denoted by $H_2^\perp$, while we use $H_2^-$ as a short for $zH_2^\perp$. We use ${\rm span}_k(x_k)$ to denote the space of all finite linear combinations of elements in $\{x_k\}_k$, a set of random variables on a suitable probability space. We also use $\X_1+\X_2=\{x_1+x_2\,:\,x_1\in \X_1, x_2\in \X_2\}$ for subspaces $\X_1$ and $\X_2$. Similarly $\sum_{j=1}^n \X_j=\{\sum_{j=1}^n x_j\,:\, x_j\in \X_j,\,j=1,\ldots, n\}$.


\section{Preliminaries}\label{sec:prel}

The forward shift $U$ is a linear operator on the Hardy space $H_2$ defined as $Uf(z)=zf(z)$. We will often identify $H_2$ and $l_2$, since they are isometric, and thus write
    \[
        U: (a_0,a_1,a_2,\ldots)\rightarrow (0,a_0,a_1,\ldots).
    \]
The backward shift $U^*$ is the adjoint operator of $U$. On $H_2$ we have $U^*f(z)=(f(z)-f(0))/z$, and accordingly we may write
    \[
        U^*: (a_0,a_1,a_2,\ldots)\rightarrow (a_1,a_2,a_3,\ldots).
    \]
Cyclic vectors of an operator $A$ are those vectors $f$ such that the closure of the span of $\{A^n f: n\ge 0\}$ is the whole space; when $f$ is not a cyclic vector (non-cyclic), the closure of the span is a proper $A$-invariant subspace.\footnote{A subspace $X$ is $A$-invariant if $AX\subset X$.} As is well known, $f\in H_2$ is cyclic for $U$ if and only if $f$ is an {\em outer} function.\footnote{Outer functions are also known in the engineering literature as {\em minimum-phase}; for definition and properties see \cite{Rud86}.}
 When this is not the case, $f$ lies in some closed invariant subspace of $U$, that is, a subspace of the form $\varphi H_2$ for some {\em inner}\footnote{Inner functions are known in the engineering literature as {\em all-pass}; for definition and properties see again \cite{Rud86}.} function $\varphi$. An invariant subspace of $U^*$ is of the form $(\varphi H_2)^\bot$. Therefore $f$ fails to be cyclic under $U^*$ if and only if it is orthogonal to one of the spaces $\varphi H_2$ with $\varphi$ inner. This is not a property that can be easily checked! A more transparent condition for failure to be cyclic with respect to $U^*$ is given by the following theorem of Douglas-Shapiro-Shields \cite{DouShaShi70}.
\begin{thm}\label{thm:DSS}
    A necessary and sufficient condition that a function $f$ in $H_2$ be $U^*$ non-cyclic is that there exists a pair of inner functions $\varphi$ and $\psi$ such that
        \[
            \frac{f}{\bar{f}}=\frac{\varphi}{\psi} ~~~\text{almost everywhere on $\mT$}.
        \]
\end{thm}

There are several easy but quite surprising properties of $U^*$ cyclic functions, see \cite{DouShaShi70}. For instance, i) a function is $U^*$ cyclic if and only if its outer factor is, and ii) if $f$ is $U^*$ cyclic and $g$ is non-cyclic, then $f+g$, $fg$ and $f/g$ are all cyclic as long as they are in $H_2$. Throughout the rest of this paper, ``cyclic'' means cyclic with respect to $U^*$ unless otherwise stated.

In $L_2$, we define $Uf(z)=zf(z)$ and its inverse as $U^{-1}f(z)=z^{-1}f(z)$. The invariant subspaces of $L_2$ for $U$ and $U^{-1}$ are slightly different to those of $H_2$ and have been extensively studied in \cite{Hel13}. The following extension of Beurling-Lax theorem~\cite{Hel13} will be needed in this paper.
\begin{prop}
If a subspace $M$ of $L_2$ is invariant for $U^{-1}$, while not for $U$, then it has the form $M=qH_2^-$ for some unimodular function $q$.
\end{prop}

The connection and correspondence between function theory on the unit disc and discrete-time, stationary stochastic processes is well known and we will make extensive use of the various facts given in the concise reference
 \cite[Chapter 10]{GreSze58}. The basis of the connection is the standard Kolmogorov isomorphism between the linear space generated by a second-order stochastic process $\{x_k \,|\, k\in \mZ \}$ and functions on $L_2$ (on the unit circle).
Throughout, in this paper, we follow the mathematical convention where $U$ (equivalently, multiplication by $z$ or $e^{i\theta}$) corresponds to unit time-delay for the corresponding process. Thus, in the usual slight abuse of notation, $z\;:\;x_k\mapsto x_{k-1}$ is the ``delay operator'' which is opposite to the way this is most often used in signal processing literature.\footnote{In signal procession, $z^{-1}$ is often used to denote delay and causal transfer functions of linear operators are analytic outside the unit disc; our convention is the opposite.}

%
%

\section{Comparison of predictor/postdictor error for a moving-average process} \label{sec:MA}

It is often suggested that for Gaussian stationary processes the time direction does not have an impact on the error variance (cf.\ \cite{DikVanTakDeg95, Wei75}). As noted earlier, {\em this is not so for multivariable processes}. We first illustrate this fact with the moving-average bivariate process defined as follows. Consider the difference equations,
\begin{subequations}\label{eq:MA}
\begin{eqnarray}
x_k&=&w_k+\alpha w_{k-1}\\
y_k&=&w_k,
\end{eqnarray}
\end{subequations}
where $\alpha\neq 0$ and the process $\{w_k\mid k\in \mZ\}$ is taken to be Gaussian, zero-mean, unit-variance and white, i.e., $E\{w_k \bar w_k\}=1$, and $E\{w_k \bar w_\ell\}=0$ for $k\neq \ell$, and consider the stochastic process
\[
\xi_k:=\left[\begin{matrix}x_k\\y_k\end{matrix}\right].
\]
We are interested in one-step ahead linear prediction.\footnote{Without loss of generality we consider only prediction of $x_0, y_0$ since all processes in this paper are stationary.} Thus, we seek to minimize the (matrix) error-variance
\[
\mE\{ (\xi_0-\hat\xi_{0|\rm past})(\xi_0-\hat\xi_{0|\rm past})^*\}
\]
in the positive-semidefinite sense. Here, $\hat\xi_{0|\rm past}$ is a function of past measurements
 $x_{-1}, x_{-2}, \ldots,$ and $y_{-1}, y_{-2},\ldots$. Since $w_0\perp x_{-\ell},y_{-\ell}$ for $\ell>0$, the solution is easily seen to be
\[
\hat\xi_{0|\rm past}=\left(\begin{array}{c}\hat x_{0|\rm past}\\ \hat y_{0|\rm past}\end{array}\right)= \left(\begin{array}{c}\alpha y_{-1}\\ 0\end{array}\right)
\]
with a corresponding (forward) error variance
\[
\min_{\hat\xi_{0|\rm past}}\! \mE\{ (\xi_0-\hat\xi_{0|\rm past})(\xi_0-\hat\xi_{0|\rm past})^*\}\!=:\Omega_{\rm f}\!=\!\left(\begin{array}{cc}1& 1\\ 1 &1\end{array}\right)\!.
\]

In the reverse time direction, since
\[
x_{k+1}-y_{k+1}=\alpha w_k,
\]
we can write the dynamics \eqref{eq:MA} as
\begin{eqnarray*}
x_{k}&=&(x_{k+1}-y_{k+1})/\alpha+\alpha w_{k-1}\\
y_{k}&=&(x_{k+1}-y_{k+1})/\alpha.
\end{eqnarray*}
Similar to the above argument for the forward time-direction, $w_{-1}$ is orthogonal to future measurements $x_{1}$, $x_{2}$,\ldots, and $y_{1}, y_{2}$,\ldots,
and hence, given future values, the optimal estimator for $x_{0}, y_{0}$ is
\begin{equation}\label{eq:Ex1optpred}\nonumber
\hat\xi_{0|\rm future}=
\left(\begin{array}{c}\hat x_{0|\rm future}\\ \hat y_{0|\rm future}\end{array}\right)= \left(\begin{array}{c}(x_{1}-y_{1})/\alpha\\ (x_{1}-y_{1})/\alpha\end{array}\right)
\end{equation}
with corresponding minimal (backward) error variance
\[
\min_{\hat\xi_{0|\rm future}} \mE\{ (\xi_0-\hat\xi_{0|\rm future})(\xi_0-\hat\xi_{0|\rm future})^*\}=:\Omega_{\rm b}
=\left(\begin{array}{cc}\alpha^2& 0\\ 0 &0\end{array}\right).
\]
The prediction problem is clearly not symmetric with respect to time, yet $\det \Omega_{\rm f}=\det \Omega_{\rm b}=0$ in agreement with the Wiener-Masani formula \cite{WieMas57}.

The above example is sufficient to underscore the dichotomy. The forward and reversed processes have similar realizations (cf.\ \cite{GeoLin14}). Indeed, we can easily see that
\begin{eqnarray*}
x_{k}&=&\alpha \tilde w_k+\tilde w_{k+1}\\
y_{k}&=&\tilde w_{k+1},
\end{eqnarray*}
is a realization for the backward process, where $\tilde w_k$ is a standard Gaussian white-noise process. The forward and backward realizations can be derived and correspond to
the left and the right analytic factors
\begin{eqnarray}
\Phi(z)&=&\left(\begin{array}{c}1+\alpha z\\1\end{array}\right)(1+\alpha z^{-1}, 1)\nonumber\\
&=&\left(\begin{array}{c}z^{-1}+\alpha \\z^{-1}\end{array}\right)(z+\alpha, z)\label{eq:SpecFac}
\end{eqnarray}
of the power spectrum $\Phi(z)$. It is possible to go one step further and construct examples where this factorization is not possible in one direction and, then, in a corresponding time-direction the process is completely deterministic.

\section{A non-reversible stochastic process}\label{sec:nonreverse}


The following example presents a situation where the power spectrum does not admit one of the two analytic factorizations and the underlying process is completely deterministic in one of the time-directions and not in the other. The stochastic process we consider
is generated by
\begin{eqnarray*}
x_k&=&w_k +\sum_{\ell=1}^\infty \frac{1}{1+\ell} w_{k-\ell},\\
y_k&=& w_{k}.
\end{eqnarray*}
The modeling filter
\[
g(z)=\sum_{\ell=0}^\infty \frac{1}{1+\ell}z^\ell
\]
for the $x_k$ component has as impulse response the harmonic series. Interestingly, while this is not a stable system in an input-output sense, when driven by a white-noise process, it generates a well-defined stochastic process with finite variance since the harmonic series is square-summable. Further, the function $g(z)$ is cyclic \cite{DouShaShi70} and, as we will see, a direct consequence is that the process is completely deterministic in the backward time-direction.

Since $w_0\perp x_{-\ell},y_{-\ell}$ for $\ell>0$, the optimal predictor is given by
\[
\hat\xi_{0|\rm past}=\left(\begin{array}{c}\hat x_{0|\rm past}\\ \hat y_{0|\rm past}\end{array}\right)= \left(\begin{array}{c}\sum_{\ell=1}^\infty \frac{1}{1+\ell} y_{-\ell} \\ 0\end{array}\right)
\]
with a corresponding (forward) error variance
\[
\inf_{\hat\xi_{0|\rm past}} \! \mE\{ (\xi_0-\hat\xi_{0|\rm past})(\xi_0-\hat\xi_{0|\rm past})^*\}\!=:\Omega_{\rm f}\!=\!\left(\begin{array}{cc}1& 1\\ 1 &1\end{array}\right)\!.
\]

In the reverse time-direction,
we estimate $x_0$, $y_0$ given future observations,
$x_\ell$, $y_\ell$, for $\ell>0$.
Since $w_\ell=y_\ell$, this is the same as estimating
$x_0$, $y_0$ given
\begin{eqnarray*}
\tilde x_k&:=&x_k-\sum_{\ell=0}^{k-2} \frac{1}{1+\ell} w_{k-\ell},\\
&=&\sum_{\ell=k-1}^\infty \frac{1}{1+\ell} w_{k-\ell},
\end{eqnarray*}
and $y_k$, for $k>0$.
Now,  
$\overline{\Span}_{k>1}\{y_k\}$ is orthogonal to $x_0,y_0, y_1$ and $\overline{\Span}_{k>0}\{\tilde x_k\}$,
and hence the estimation problem above is equivalent to estimating $x_0,y_0$ based on $y_1$ and $\tilde x_k$ for $k>0$. In fact, $y_1$ is not needed and as we will see, $x_0,y_0$ can be predicted with arbitrary precision based only on $\tilde x_k$ for $k>0$.
The relation between $\tilde x_k$ and $w_{k}$ for $k\in\mZ$ can be expressed as
\begin{equation}\label{eq:Hilbert}
{\bf \tilde x}:=\arraycolsep=2pt\def\arraystretch{1.4}
\left(\begin{array}{c}\tilde x_1\\\tilde x_2\\\vdots\\\tilde x_n\\\vdots\end{array}\right)=\left(\begin{array}{cccc}
1&\frac{1}{2}&\frac{1}{3}&\cdots\\
\frac{1}{2}&\frac{1}{3}&\frac{1}{4}&\cdots\\
\vdots&\vdots&\vdots&\\
\frac{1}{n}&\frac{1}{n+1}&\frac{1}{n+2}&\cdots\\
\vdots&\vdots&\vdots&
\end{array}\right)
\left(\begin{array}{c}w_1\\w_{0}\\w_{-1}\\\vdots\end{array}\right)=:{\mathcal H} {\bf w},
\end{equation}
where ${\mathcal H}$ denotes the (infinite) Hilbert matrix, or equivalently, the representation of a Hankel operator with symbol the harmonic series.
Note that the $(k+1)$th row of ${\mathcal H}$ corresponds to the backward shifted input responses
\[
\sum_{\ell=0}^\infty \frac{1}{k+\ell+1}z^\ell = U^{*k} g(z), \quad \mbox{ for } k=0, 1, \ldots,
\]
%
and since $g(z)$ is cyclic, we have $\overline{\Span}_{k\geq 0}\{U^{*k} g(z)\}=H_2$. 
Combining this and the Kolmogorov isomorphism we deduce that any linear combination of $\bf w$ with finite norm can be approximated by a finite linear combination of ${\bf \tilde x}$ with arbitrarily small error.
The infimum of the backward error variance is therefore
\[
\inf_{\hat\xi_{0|\rm future}} \mE\{ (\xi_0-\hat\xi_{0|\rm future})(\xi_0-\hat\xi_{0|\rm future})^*\}=:\Omega_{\rm b}
=\left(\begin{array}{cc}0& 0\\ 0 &0\end{array}\right)
\]
and the time series $\{\xi_k\}$ is uniquely determined by the infinite future (cf.\ \cite{GreSze58, LinPic15}).

A alternative path to arrive at the same conclusion and deduce that the backward prediction error variance is zero can be based on well-known properties of the Hilbert matrix in \eqref{eq:Hilbert}.
The Hilbert matrix does not have a discrete spectrum \cite{magnus1950spectrum}; see \cite{Cho83} for an elementary proof. Therefore, its range is dense \cite{friedman1970foundations}. That is, for any vector ${\bf r}=(r_1, r_0, r_{-1}, \ldots)^T \in l_2$, there is a vector ${\bf a}=(a_1,a_2,\ldots)^T\in l_2$, with a finite number of nonzero elements, making the difference
\begin{equation*}\label{eq:Hilbert2}
\|{\bf r}^T- {\bf a}^T{\mathcal H}\|_2
\end{equation*}
arbitrarily small. Thus, any linear combination of elements $w_{k}$ for $k\le 0$, namely ${\bf r}^T{\bf w}$, can be approximated by ${\bf a}^T{\bf \tilde x}$ with arbitrary precision. Therefore, any element $x_k, y_k, w_k$ with $k\in\mZ$ is either known or can be predicted with arbitrary precision.


\section{Characterization of backward deterministic rank-one processes} \label{sec:CharRank1}
An n-dimensional Gaussian stochastic process is regular if its spectrum admits a (right) analytic factorization (see e.g., \cite{Roz60}), and hence may be represented  in the form
\begin{equation*}
\xi_k= \sum_{\ell=0}^\infty G_{\ell}w_{k-\ell}
\end{equation*}
where $w_k$ is a white-noise process, and the sequence $\{G_k\}_{k\ge 0}\in l_2^n$
\cite{Roz60,WieMas57}. Rank-one regular processes are those where the white-noise process $w_k$ may be taken as a scalar process.

Building on the example from the previous section and using results from \cite{DouShaShi70}, we characterize all the regular rank-one processes that are backward deterministic. We start by identifying a subclass of bivariate processes that contains the example from Section~\ref{sec:nonreverse}. Below we take $G_\ell=(g_\ell,h_\ell)^T\in{\mathbb C}^2$, that is $g_\ell,h_\ell\in\mathbb C$.
\begin{thm}\label{thm:pcnc}
Consider the stochastic processes
\begin{subequations}\label{eq:SP}
\begin{eqnarray}
x_k&=&\sum_{\ell=0}^\infty g_\ell w_{k-\ell},\\
y_k&=& \sum_{\ell=0}^\infty h_\ell w_{k-\ell}
\end{eqnarray}
\end{subequations}
where $g(z)=\sum_{\ell=0}^\infty g_\ell z^\ell$ is cyclic and $h(z)=\sum_{\ell=0}^\infty h_\ell z^\ell\neq 0$ is non-cyclic. Then the backward process is deterministic.
\end{thm}

To show this, we need the following lemma.
\begin{lemma} \label{lm:hminus}
If $h$ is a non-cyclic function, then there exists an inner function $\psi$ such that
\[
\overline{\Span}_{k> 0}(z^{-k}\psi  h)\supset H_2^\perp.
\]
\end{lemma}
\begin{proof}
See the Appendix.
\end{proof}

\begin{proof}[Proof of Theorem \ref{thm:pcnc}]
Let the inner function $\psi$ be selected according to Lemma \ref{lm:hminus} so that $\overline{\Span}_{k> 0}(z^{-k}\psi h)\supset H_2^\perp$.
The backward prediction error for $x_0$ is bounded by
\begin{eqnarray}
\sqrt{\mE(\|x_0-\hat x_{0|{\rm future}}\|^2)}&=& \|a^*g+b^*h\|_{2}\nonumber\\
&=&\|\psi (a^*g+ b^*h)\|_{2}\nonumber\\
&\le&\|P_{H_2}(\psi a^*g)\|_{2} \label{eq:perror}
\\
&&+\|\psi b^*h +P_{H_2^\perp}(\psi a^*g)\|_{2}\nonumber
\end{eqnarray}
where $a$ and $b$ are polynomials with $a(0)=1$ and $b(0)=0$, corresponding
to the predictor
$\hat x_{0|{\rm future}}=-\sum_{\ell=1}^n(\bar a_\ell x_\ell+\bar b_\ell y_\ell)$.
The inequality in \eqref{eq:perror} follows from the triangle inequality.
Since $g$ is cyclic, so is $U^*(g\psi)$, hence $\Span_{\ell\ge 1} U^{*\ell} \psi g$ is dense in $H_2$, and therefore the first term of \eqref{eq:perror} which equals
\begin{eqnarray*}
\|P_{H_2}(\psi a^*g)\|_{2}=\|\psi g+\sum_{\ell=1}^\infty \bar a_{\ell}U^{*\ell}(\psi g)\|_2,
\end{eqnarray*}
can be made arbitrarily small by selecting the polynomial $a$ properly.  Since $\overline{\Span}_{k> 0}(z^{-k}\psi h)\supset H_2^\perp$, the polynomial
$b$ can be selected so that the second term of \eqref{eq:perror} is arbitrarily small as well.

A similar argument can be used to show that $y_0$ can be estimated with arbitrarily small error, by considering polynomials with $a(0)=0$ and $b(0)=1$ in \eqref{eq:perror}. This completes the proof.
\end{proof}



Following the same lines as in the proof of Theorem \ref{thm:pcnc}, one can in fact show that
\[
\overline{\Span}_{k> 0}\{z^{-k} g\}+\overline{\Span}_{k> 0}\{z^{-k} h\}=L_2,
\]
and therefore, the backwards prediction error is zero as a result of the Kolmogorov isomorphism. In this case the input sequence $w_k,$ for $k\in \mZ,$ may be reconstructed arbitrarily well from the future output, $x_k, y_k$ for $k> 0$.

As stated in Theorem \ref{thm:DSS} \cite{DouShaShi70}, a function $g\in H_2$ is cyclic if and only if $g/\bar g$ belong to $\mathcal{J}$, the set of unimodular functions that are quotients of inner functions,\footnote{The set $\mathcal{J}$ is dense in the set of all unimodular functions with respect to $L_2$-norm. See \cite{DouRud69} for more discussion on $\mathcal{J}$.} i.e.,  $\mathcal{J}=\{\varphi/\psi\,:\, \varphi, \psi \mbox{ inner}\}$.
This result is central to our characterization of backward deterministic rank-one processes, and
 leads to our main result.

\begin{thm}\label{thm:2GenDet}
Let $g,h\in H_2$, then the following conditions are equivalent
\begin{itemize}
\setlength{\itemsep}{1pt}
\setlength{\parskip}{1pt}
\item[{(a)}] The system \eqref{eq:SP} is backward deterministic,
\item[{(b)}] $
\overline{\Span}_{k\ge 0}\{z^{-k} g\}+\overline{\Span}_{k\ge 0}\{z^{-k} h\}=L_2$,
\item[{(c)}]$
g\bar h/(\bar g h)\notin \mathcal{J}.$
\end{itemize}
\end{thm}
\begin{proof}
See the Appendix.
\end{proof}
Note that Theorem \ref{thm:pcnc} follows as a special case, from the equivalence between (a) and (c) and by using the fact that $g/\bar g \notin \mathcal{J}$ and $ h/\bar h \in \mathcal{J}$ (see Theorem \ref{thm:DSS}).

In view of Theorem \ref{thm:2GenDet}, we define backward deterministic processes generated by a set of functions as follows.
\begin{defn}
The functions $g^{(1)},g^{(2)}, \ldots, g^{(n)}\in H_2$ are called backward deterministic if
\begin{equation}
\sum_{j=1}^n \overline{\Span}_{k\ge 0}\{z^{-k} g^{(j)}\}=L_2.
\end{equation}
\end{defn}

As a corollary to Theorem \ref{thm:2GenDet} we also get an analogous result for general  vector-valued rank-one processes.
\begin{cor}\label{cor:multi}
The non-zero functions $g^{(1)},g^{(2)}, \ldots, g^{(n)}\in H_2$ are backward deterministic if and only if
\begin{equation}\label{eq:gh2}
\frac{g^{(1)}}{\bar g^{(1)}}\frac{\bar g^{(j)}}{ g^{(j)}}\notin \mathcal{J}
\end{equation}
for some $j=2,\ldots, n$.
\end{cor}

\begin{proof}
See the Appendix.
\end{proof}

\section{Concluding remarks}

While the moving-average process in Section \ref{sec:MA} admits a spectral factorization for the backward process \eqref{eq:SpecFac}, there is no such factorization for the non-reversible processes in Section \ref{sec:nonreverse} and Section \ref{sec:CharRank1}.
This may be viewed in the context of decompositions for stochastic processes that are Gaussian, zero mean, and stationary. Namely, the Hilbert space generated by any such process may be decomposed as
\[
{\bf H}(\xi_k)={\bf H}_{-\infty}(\xi_k)\, \oplus {\bf H}(w_k),
\]
in terms of the remote past and the driving noise, namely,
\begin{eqnarray*}
{\bf H}(\xi_k)&=&\overline{\Span}_{k\in \mZ} \{\xi_k\}\\
{\bf H}_{-\infty}(\xi_k) &=&\cap_{t\in \mZ}\overline{\Span}_{k\le t} \{\xi_k\},\;\mbox{ and }\\
{\bf H}(w_k)&=&\overline{\Span}_{k\in \mZ} \{w_k\},
\end{eqnarray*}
and similarly in terms of the remote future and the Hilbert space generated in the backward direction by a driving noise
\[
{\bf H}(\xi_k)={\bf H}_{+\infty}(\xi_k)\, \oplus {\bf H}(\bar w_k)
\]
(see \cite[Section 4.5]{LinPic15} for details). The process is reversible if
the remote past and the remote future coincide.
Here we have considered concrete examples of a non-reversible processes where the remote past is trivial while the remote future spans the entire process.

The essence of these examples (Section \ref{sec:nonreverse}-\ref{sec:CharRank1} and, cf.\ \cite{LinPic15}) is that the power spectrum of $\{\xi_k\}$, being
\[
\left(\begin{array}{cc}g(z)\\ h(z)\end{array}\right)\left(\begin{array}{cc}g(z)^* & h(z)^*\end{array}\right)
\]
fails to have a co-analytic spectral factorization, or equivalently, the backward process is not regular \cite{Roz60}. This can be shown using Theorem \ref{thm:DSS} (see also \cite{LinPic15}). It also fails to satisfy condition $3$ of Theorem $2$ in \cite{Roz60}. This absence of co-analytic factorization renders the backward process deterministic.

It is quite apparent that the issues herein are quite technical in nature. Yet, they impact in significant ways the relevance of certain models for time-series. Indeed, existence of left and right analytic factorizations for the corresponding power spectra may fail and, in such cases, the dichotomy between past and future becomes central. On the other hand, when the power spectrum is coersive (nonsingular at every frequency), it admits both left and right analytic spectral factors and the issue become mute. But even so, the limiting case where the power spectrum seizes to be coersive and the dichotomy appears, requires further understanding from a practical standpoint. In particular, it is of interest to understand the interplay between the scales and window lengths required to estimate the present from past and future values, respectively. Deeper insights on how prediction and postdiction in stochastic models relate to time directionality, causality, ergodicity and even mixing may also result in further exploration of these issues.

\section{Acknowledgment}

Insights on the subject by Anders Lindquist, Alexandre Megretski and Sergei Treil are gratefully acknowledged.

\appendix

\subsection{Proof of Lemma \ref{lm:hminus}}
Since $h$ is non-cyclic there exists an inner function $\varphi$ such that $\Span_{k> 0} U^{*k}h= (\varphi H_2)^\perp_{H_2}$, hence
\begin{eqnarray*}
\overline{\Span}_{k> 0}\;z^{-k}h&\subset& (\varphi H_2)^\perp_{L_2}\\
\overline{\Span}_{k> 0}\;z^{-k}h\bar \varphi z&\subset&  H_2^-.
\end{eqnarray*}
Since the left hand side is invariant with respect to $z^{-1}$ it is on the form $\bar \psi H_2^-$ where $\psi$ is inner.
From Beurling-Lax theorem \cite{Hof07,Hel13} it follows that
\begin{eqnarray*}
\overline{\Span}_{k> 0}\;z^{-k}h\bar \varphi z&=&  \bar \psi H_2^-\\
\overline{\Span}_{k> 0}\;z^{-k}h\psi &\supset &\overline{\Span}_{k> 0}\;z^{-k}h\psi \bar \varphi=  H_2^\perp.
\end{eqnarray*}

\subsection{Proof of main theorem (Theorem \ref{thm:2GenDet})}

In order to prove the main theorem we will use the following lemmas.

\begin{lemma}\label{lm:qchar}
For any function $g\in H_2$, the subspace $\overline{\Span}_{k\ge 0}\{z^{-k} g\}$ is equal to $qH_2^-$, where $q=g/\bar g_{\rm outer}$ and $g_{\rm outer}$ is the outer part of $g$.
\end{lemma}

\begin{proof}
Clearly $M=\overline{\Span}_{k\ge 0}\{z^{-k} g\}\subset L_2$ is a invariant subspace for $U^{-1}$ while not for $U$, so it has the form $M=qH_2^-$ for some unimodular function $q$ and $M=q\oplus z^{-1}M$.  The function $q$ is determined by the subspace up to a constant factor. We next compute one such $q$. Since $q\in M$, we have that $q=g\bar{f}$ for some analytic function $f$. We claim that one feasible $f$ is given by
 \[
        f=1/g_{\rm outer},
    \]
where $g_{\rm outer}$ is the outer factor of $g$. Note $q=g\bar{f}=g/\bar{g}_{\rm outer}$ is indeed a unimodular function. To see $M=qH_2^{-}$, it is enough to show $q\perp z^{-1}M$, which is equivalent to $q\perp z^{-k}g$ for all $k\ge 1$. This follows from
    \[
        (q,z^{-k}g)=(g\bar{g}/\bar{g}_{\rm outer},z^{-k})=(g_{\rm outer},z^{-k})=0,~k\ge 1,
    \]
which completes the proof.
\end{proof}

\begin{lemma}
A function $g\in H_2$ is non-cyclic if and only if $q\in \mathcal{J}$, where $q$ is a unimodular function satisfying $qH_2^-=\overline{\Span}_{k\ge 0}\{z^{-k} g\}$.
\end{lemma}

\begin{proof}
Lemma \ref{lm:qchar} implies that the unimodular function $q$ is on the form
    \[
        q=\lambda\frac{g}{\bar{g}_{\rm outer}}
    \]
for some constant $|\lambda|=1$.
A function $g\in H_2$ is non-cyclic if and only if there exists a pair of inner functions $\varphi$ and $\psi$ such that
    \[
        \frac{g}{\bar{g}}=\frac{\varphi}{\psi}~~~\text{almost everywhere on $\mT$}.
    \]
By combining these two facts, it follows that if $g\in H_2$ is non-cyclic, then
    \[
        q=\frac{\lambda g}{\bar{g}g_{\rm inner}}=\frac{\lambda\varphi}{\psi g_{\rm inner}}\in \mathcal{J}.
    \]
Conversely, if $q\in \mathcal{J}$, then
    \[
        \frac{g}{\bar{g}}=\bar\lambda q g_{\rm inner}=\frac{\varphi}{\psi}
    \]
for some inner functions $\varphi,\psi$,
and hence $g$ is non-cyclic.
\end{proof}

\begin{lemma}\label{lm:subspace}
Let $q_1$ and $q$ be unimodular functions, then $q_1H_2^-\subset qH_2^-$ if and only if $\varphi q_1=q$ for some inner function $\varphi$.
\end{lemma}

\begin{proof}
The sufficiency follows from the fact that any $g\in q_1H_2^-$ is on the form $q_1 \bar f$ for some $f\in H_2$, and hence satisfies $g=q \bar f\bar\varphi\in qH_2^-$. To see the necessity, we note that $q_1H_2^-\subset qH_2^-$ implies $q_1\in qH_2^-$. It follows that $q_1=q\bar\varphi$ for some unimodular $\varphi\in H_2$, that is, $\varphi$ is a inner function. This completes the proof.
\end{proof}

\begin{lemma} \label{lm:subspaces}
Let $q_1H_2^{-}$ and $q_2H_2^{-}$ be subspaces of $L_2$ where $q_1$ and $q_2$ are unimodular, then $q_1H_2^{-}+q_2H_2^{-}=L_2$ holds if and only if $q_1/q_2\notin \mathcal{J}$.
\end{lemma}


\begin{proof}
We use proof by contradiction. Assume first that $q_1/q_2\in \mathcal{J}$, i.e., there exist inner functions $\psi_1, \psi_2$ such that $q_1/q_2=\psi_2/\psi_1$. Now, let $q$ be the unimodular function $q=q_1\psi_1=q_2\psi_2$. Then by Lemma \ref{lm:subspace} we have  $q_jH_2^-\subset qH_2^-$ for $j=1,2$, and by linearity it follows that
\begin{equation*}
q_1H_2^-+q_2H_2^-\subset qH_2^-\neq L_2.
\end{equation*}
Note that $qH_2^-\neq L_2$ holds since, e.g.,  $qz\notin qH_2^-$.

Conversely, assume that $q_1H_2^-+q_2H_2^- \neq L_2$, then $q_1H_2^-+q_2H_2^-\subsetneq L_2$ is an invariant subspace for $U^{-1}$ while not for $U$ (this follows since it contains an analytic function \cite{Hel13}). As a consequence of this there is an unimodular function $q$ such that $q_1H_2^-+q_2H_2^-=qH_2^-$. This implies that
$q_jH_2^-\subset qH_2^-$, for $j=1,2$. By Lemma \ref{lm:subspace} there exists inner functions $\psi_1, \psi_2$ such that $q=q_1\psi_1=q_2\psi_2$, and hence $q_1/q_2\in \mathcal{J}$.
\end{proof}

\begin{proof}[Proof of Theorem \ref{thm:2GenDet}]
The equivalence between (a) and (b) follows directly from the Kolmogorov isomorphism.
Using Lemma \ref{lm:qchar}, it follows that (b) is equivalent to
\begin{equation}\label{eq:q1q2}
q_1H_2^-+q_2H_2^-=L_2,
\end{equation}
where $q_1=g/\bar g_{\rm outer}$ and $q_2=h/\bar h_{\rm outer}$.
By Lemma \ref{lm:subspaces}, Equation \eqref{eq:q1q2} holds if and only if
$q_1/q_2\notin \mathcal{J}$.
Since $q_1/q_2=g\bar hh_{\rm inner}/(\bar ghg_{\rm inner})$, where $g_{\rm inner}, h_{\rm inner}$ are the inner parts of $g$ and $h$ respectively, the equivalence with
Theorem \ref{thm:2GenDet} (c) follows.
\end{proof}

\subsection{Proof of Corollary \ref{cor:multi}}
If ${g^{(1)}}{\bar g^{(j)}}/({ g^{(j)}}{\bar g^{(1)}})\notin \mathcal{J}$ for some $j=2,\ldots, n$, then from Theorem \ref{thm:2GenDet} we know $\overline{\Span}_{k> 0}\{z^{-k} g\}+\overline{\Span}_{k> 0}\{z^{-k} h\}=L_2$, which implies that $\sum_{j=1}^n \overline{\Span}_{k\ge 0}\{z^{-k} g^{(j)}\}=L_2$.

Conversely, suppose ${g^{(1)}}{\bar g^{(j)}}/({ g^{(j)}}{\bar g^{(1)}})\in \mathcal{J}$ for all $j=2,\ldots, n$, then there exists inner functions $\varphi_j, \psi_j$ for each $j$ such that
	\[
		\frac{g^{(1)}}{\bar g^{(1)}}\frac{\bar g^{(j)}}{ g^{(j)}}=\frac{\varphi_j}{\psi_j},
	\]
which is equivalent to
	\[
		\frac{q_1}{q_j}=\frac{\varphi_j{g}_{\rm inner}^{(j)}}{\psi_j{g}_{\rm inner}^{(1)}}, ~~j=2,\ldots, n
	\]
by Lemma \ref{lm:qchar}. Here $q_i$ are unimodular functions such that $\overline{\Span}_{k\ge 0}\{z^{-k} g^{(i)}\}=q_iH_2^-$ for all $i=1,\ldots, n$. Now let $q=q_1{g}_{\rm inner}^{(1)} \prod_{j=2}^{n} \psi_j$, then $q$ is unimodular function and satisfies that $q/q_i$ is inner for each $i=1,\ldots, n$, which implies $q_i H_2^- \subset qH_2^-$ by Lemma  \ref{lm:subspace}. As a result, we conclude that
	\[
		\sum_{i=1}^n \overline{\Span}_{k\ge 0}\{z^{-k} g^{(i)}\}=\sum_{i=1}^n q_iH_2^-\subset qH_2^-\neq L_2,
	\]
which leads to a contradiction. This completes the proof.

\bibliographystyle{IEEEtran}
\bibliography{refs}

\end{document}